\documentclass[12pt]{amsart}

\usepackage{graphicx}
\usepackage{eucal}
\usepackage{xcolor}
\usepackage{amsfonts,amssymb,amscd}

\newtheorem{theorem}{Theorem}[section]
\newtheorem{lemma}[theorem]{Lemma}
\theoremstyle{definition}
\newtheorem{definition}[theorem]{Definition}

\newtheorem{corollary}[theorem]{Corollary}
\newtheorem{observation}[theorem]{Observation}

\theoremstyle{remark}

\numberwithin{equation}{section}
\usepackage{graphicx}
\usepackage{hyperref}
\usepackage{geometry}
\begin{document}
\title[On the Blaschke rolling disk theorem] {\sc On the Blaschke rolling disk theorem}
\author{Jos\'e Ayala}
\address{Facultad de Ciencias, Universidad Arturo Prat, Iquique, Chile} 
\address{School of Mathematics and Statistics, University of Melbourne
              Parkville, VIC 3010 Australia}
              \email{jayalhoff@gmail.com}
%\date{t}
%\baselineskip=20 true pt
%\maketitle \baselineskip=1\normalbaselineskip
\baselineskip=20 true pt
\maketitle \baselineskip=1.10\normalbaselineskip
%52A10  	Convex sets in $2$ dimensions (including convex curves) 
%53A04  	Curves in Euclidean space
%53C42  	Immersions (minimal, prescribed curvature, tight, etc.)

%----------------------------------------------------------------------------------------
%	ABSTRACT
%----------------------------------------------------------------------------------------
\begin{abstract} 
\baselineskip=20 true pt
\maketitle \baselineskip=1.15\normalbaselineskip

The Blaschke rolling disk theorem is a classical inclusion principle in differential geometry. This states that a planar convex domain whose boundary is a curve of class $C^2$ with (signed) curvature not exceeding a positive constant $\kappa$ is such that for each point on its boundary there exists a disk of radius ${1}/{\kappa}$ tangent to the boundary included in the closure of the domain. 

We describe geometric conditions relying exclusively on curvature and independent of any kind of convexity that allows us to give necessary and sufficient conditions for the existence of rolling disks for planar domains that are not necessarily convex. We finish by presenting an algorithm leading to a decomposition of any planar domain into a finite number of maximal rolling regions.
\end{abstract} 

%----------------------------------------------------------------------------------------
%	INTRODUCTION
%----------------------------------------------------------------------------------------
\section{Introduction} 

The Blaschke rolling disk theorem is an inclusion principle first proven in 1916 by Wilhelm Blaschke \cite{blaschke}. It states that a convex domain $\mathcal K\subset \mathbb R^2$ whose boundary $\partial \mathcal K$ is a curve of class $C^2$ with curvature not exceeding a constant $\kappa>0$ is such that for each point in $\partial \mathcal K$ there exists a disk of radius ${1}/{\kappa}$ tangent to $\partial \mathcal K$ included in the closure of $\mathcal K$. 

The convexity of $\mathcal K$ implies that the curvature of $\partial \mathcal K$ is signed. That is,  the curvature of $\partial \mathcal K$ may be nonnegative or nonpositive but not both. After dropping the convexity of $\mathcal K$, the curvature of $\partial \mathcal K$ may change in sign. Accordingly, we assume that the absolute curvature of $\partial \mathcal K$ does not exceeds $\kappa>0$. Note the Blaschke rolling disk theorem is a local to global result since a local property as the bound on absolute curvature on $\partial \mathcal K$ allow us to conclude about a global property of $\mathcal K$ as the existence of a radius $1/\kappa$ disk tangent at each point at $\partial \mathcal K$ included in the closure of $\mathcal K$. 
 
By describing geometric conditions derived from the curvature of $\partial \mathcal K$ we find necessary and sufficient conditions for the existence of rolling disks for general domains $\mathcal K\subset \mathbb R^2$. Definition \ref{dfnarc} is the key feature to identify the existence of arcs in $\partial \mathcal K$ whose end points, called essential and inessential terminals, guarantee the existence of several geometric objects. In particular, the existence of parallel tangents, who are key to characterise convexity for plane curves.

The main contributions of this paper are: Theorem \ref{halfdisk} which is a `half version' of the Pestov-Ionin theorem \cite{pestov} in which we prove the existence of a half disk for curves that are not close. Theorem \ref{nonconvex} where we characterise convexity and non convexity of a domain $\mathcal K\subset \mathbb R^2$ in terms of the type of terminals its boundary $\partial \mathcal K$ admits. Theorem \ref{main} gives necessary and sufficient conditions for a planar domain $\mathcal K$ to be internal rolling, external rolling or both simultaneously. In Theorem \ref{dcp} we provide an algorithm leading to a decomposition of any planar domain into a finite number of maximal rolling regions. As a consequence of Theorem \ref{main}, we obtain an updated version for the Blaschke rolling disk theorem. The methods here presented allow natural adaptations in higher dimensions.

%----------------------------------------------------------------------------------------
%	THE TURN CONDITION
%----------------------------------------------------------------------------------------
\section{The turning condition} \label{arc}

We consider a loop $\sigma: S^1\to\mathbb R^2$ to be the homeomorphic image of the circle. The loops are of class $C^1$, regular, arc length parametrised, traversed clockwise with period equals to its length. A loop is called {$\kappa$-constrained} if is of class piecewise $C^2$ and has {bounded absolute curvature} i.e. $||\sigma''||\leq \kappa$, when defined, with $r:={1}/{\kappa}>0$ a constant. Therefore, $\kappa$-constrained loops have almost everywhere well defined curvature with respects to its arc length parameter. We often refer to a loop instead to a $\kappa$-constrained loop. 
%We refer to the first branch of $S^1$.

%We consider $C^1$ immersed loops $\gamma: S^1\to\mathbb R^2$, where the length of the derivative $||\gamma^\prime||=1/\ell$ and $S^1=[0,1]/\{0\sim 1\}$, denoting the length of the loop by $\ell$. We endow the space of immersed loops with the $C^1$ metric.

We endow the space of $\kappa$-constrained loops with the $C^1$ metric. The ambient space of the loops is the euclidean plane with the topology induced by the euclidean metric.  The interior, boundary, closure and complement of a set $\mathcal K\subset \mathbb R^2$ are denoted by $int(\mathcal K)$, $\partial \mathcal K$, $cl (\mathcal K)$ and $\mathcal K^c$ respectively. A planar domain is considered to be an open connected set in $\mathbb R^2$. We regard a loop as either the map $\sigma: S^1\to\mathbb R^2$ or its image  when no confusion arises. 

The Jordan curve theorem asserts that the complement of a loop in $\mathbb R^2$ corresponds to two regions each of which has boundary the loop. One region is a topological disk enclosed by the loop and the other is unbounded \cite{docarmo}.

%----------------------------------------------------------------------------------------
%	DEFN ROLLING DISKS
%----------------------------------------------------------------------------------------
\begin{definition}\label{dfnrolling}
Let $T(t)$ and $N(t)$ be a continuous choice of tangent and normal line at $\sigma(t)$, $t \in S^1$. A domain $\mathcal K\subset \mathbb R^2$ is called {internal rolling} if its boundary $\partial \mathcal K=\sigma$ is a $\kappa$-constrained loop such that for all $t\in S^1$ there exist disk $D(t)$ of radius $r={1}/{\kappa}$ tangent to $T(t)$ with $D(t)\subset cl({\mathcal K})$. A domain $\mathcal K\subset \mathbb R^2$ is called {external rolling} if its boundary $\partial \mathcal K=\sigma$ is a $\kappa$-constrained loop such that for all $t\in S^1$ there exist disk $E(t)$ of radius $r={1}/{\kappa}$ tangent to $T(t)$ with $E(t)\subset cl({\mathcal K}^c)$. A domain $\mathcal K\subset \mathbb R^2$ is called {rolling} if it is internal and external rolling, see Fig. \ref{exroll}.
 \end{definition}

%----------------------------------------------------------------------------------------
%	FIGURE 1
%----------------------------------------------------------------------------------------
\begin{figure}[h]
	\centering
	\includegraphics[width=1\textwidth,angle=0]{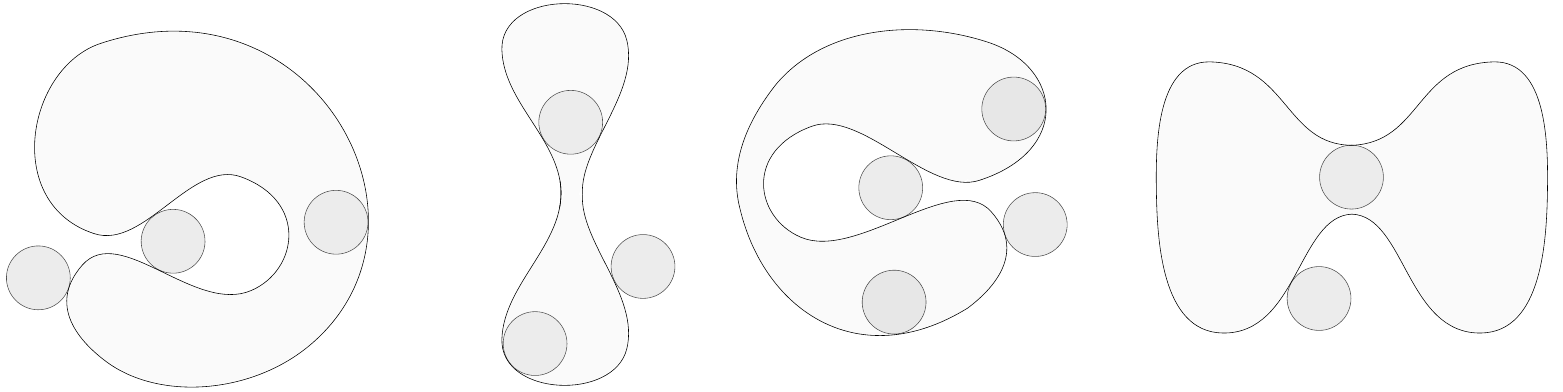}
\caption{Examples of planar domains and tangent disks of radius $r=1/\kappa$. From left to right.  A domain that is internal rolling but not external rolling. A domain that is external rolling but not internal rolling. A domain that is neither internal nor external rolling. A domain that is rolling. Clearly, none of these domains is convex.}
	\label{exroll}
\end{figure}

Next we rewrite the well known Blaschke rolling disk theorem whose proof in its original form can be found in  \cite{blaschke}.

\vspace{.3cm}
\noindent {\bf Theorem} A convex domain $\mathcal K\subset \mathbb R^2$ with boundary of class $C^2$ and curvature not exceeding a positive constant is internal rolling. 
\vspace{.3cm}

The definition below is based on the classification of the homotopy classes of $\kappa$-constrained curves \cite{papere}. Roughly speaking $\kappa$-constrained curves satisfy the same properties as the loops but they are the image of a nondegenerate close interval. In Observation \ref{connect} we explain the connections between the results in \cite{papere} and Definition \ref{dfnarc}.

%----------------------------------------------------------------------------------------
%	DEFN TURN CONDITION
%----------------------------------------------------------------------------------------
\begin{definition}\label{dfnarc}(Turning condition). Let $\sigma: S^1\to\mathbb R^2$ be a $\kappa$-constrained loop. Consider $\sigma: [t_1,t_2]\subset S^1\to\mathbb R^2$. Let $D_1, D_2$ be radius $r={1}/{\kappa}$ disks with $\sigma(t_1),\sigma(t_2) \in \partial D_1\cap \partial D_2$.
\noindent A short curve satisfies:

\begin{itemize}
\item $||\sigma(t_1)-\sigma(t_2)||\leq2r$;
\item $\sigma([t_1,t_2])\subset D_1\cap D_2 $.
\end{itemize}

\noindent A long curve satisfies:

\begin{itemize}
\item $||\sigma(t_1)-\sigma(t_2)||< 2r$;
\item $\sigma([t_1,t_2])$ has a point in $ (D_1\cap D_2)^c$. 
\end{itemize}

A curve is called simple if $||\sigma(t_1)-\sigma(t_2)||\geq 2r$. By definition an arc of a circle of radius $r$ is short, simple but never long. We refer to the points $\sigma(t_1),\sigma(t_2)$ as terminals. Terminals are called essential if they are not terminals of a short or a simple curve. Terminals that are not essential are called inessential, see Fig \ref{exarcs}. A curve is said to be essential or inessential depending on its terminals. Note that a pair of terminals separate a loop intro two complementary curves.
\end{definition}

%----------------------------------------------------------------------------------------
%	FIGURE 2
%----------------------------------------------------------------------------------------
\begin{figure}[h]
	\centering
	\includegraphics[width=1\textwidth,angle=0]{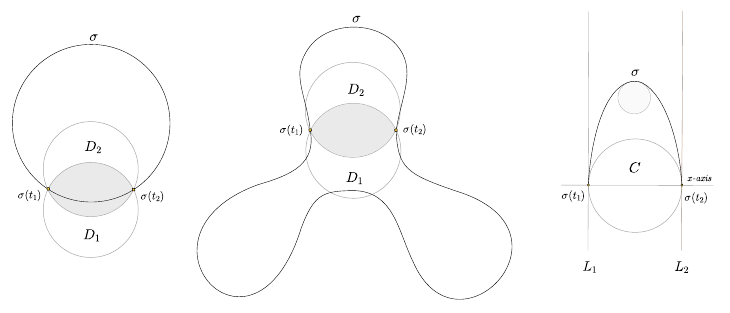}
\caption{Left: a $\kappa$-constrained loop with a pair of inessential terminals is decomposed into two curves one long and one short. Short curves are always included in $D_1\cap D_2$. Centre: a $\kappa$-constrained loop with a pair of essential terminals is decomposed into two long curves. Long curves always have a point in $ (D_1\cap D_2)^c$. Right: a $\kappa$-constrained curve whose terminals are distant apart less or equal to $2r$ cannot make a u-turn while defined in an open band of width $2r$.}
	\label{exarcs}
\end{figure}

%IS THIS A COROLLARY?
%Stadium loops and ellipses only admit inessential terminals.

Before we explain Definition \ref{dfnarc} we would like to highlight the importance of the following result from which many results derive. In its original form, this establishes that a $\kappa$-constrained curve whose terminals are distant apart less than $2r$ cannot make a u-turn while defined in the open band of width $2r$, see Fig. \ref{exarcs} right. We present an updated version of this result for the case when the terminals are distant apart less or equal to $2r$. The proof remains identical to the one presented in  \cite{papere}. 

%----------------------------------------------------------------------------------------
%	FUNDAMENTAL LEMMA
%----------------------------------------------------------------------------------------
\begin{lemma}\label{r1}(Lemma 3.1 \cite{papere}). A $\kappa$-constrained curve $\sigma: [t_1,t_2] \rightarrow {\mathcal B}$ where,
$${\mathcal B}=\{(x,y)\in{\mathbb R}^2\,|\, -r< x< r\,\,,\,\,y\geq 0 \}$$
cannot satisfy both:
\begin{itemize}
\item $\sigma(t_1),\sigma(t_2)$ are points on the $x$-axis;
\item If $C$ is a radius $r$ circle with centre on the nonpositive $y$-axis and $\sigma(t_1),\sigma(t_2)\in C$, then some point in $\sigma([t_1,t_2])$ lies above $C$.
\end{itemize}
\end{lemma}

Two $\kappa$-constrained curves with fixed common terminals are said to be $\kappa$-constrained homotopic if at each stage of the deformation the intermediate curves are $\kappa$-constrained curves while keeping the terminals fixed. Next we explain our main results in \cite{papere} and put Definition \ref{dfnarc} into context. 

%----------------------------------------------------------------------------------------
%	OBS 1
%----------------------------------------------------------------------------------------
\begin{observation} \label{connect} Let $D_1,D_2$ be radius $r$ disks. Let $\sigma: [t_1,t_2]\to\mathbb R^2$ be a $\kappa$-constrained curve with (fixed) terminals satisfying $0<||\sigma(t_1)-\sigma(t_2)||<2r$. Then, there are exactly two connected components of $\kappa$-constrained curves connecting these terminals. One is an isotopy class of short curves. All the curves in this class are confined to be in $D_1\cap D_2$ where $\sigma(t_1),\sigma(t_2) \in \partial D_1\cap \partial D_2$, see Theorem 4.21 in \cite{papere}. The other, is a homotopy class of long curves connecting the same terminals. These curves are such that they are $\kappa$-constrained homotopic to a curve with a point in $( D_1\cap D_2)^c$, see Theorem 4.22 in \cite{papere}. 

Accordingly for fixed terminals, a curve cannot be $\kappa$-constrained homotopic to a short and a long arc simultaneously, and therefore it cannot be short and long simultaneously. 

In addition, if $||\sigma(t_1)-\sigma(t_2)||\geq 2r$ there is a single connected component of simple curves, similarly as for homotopies between continuous plane paths. If $\sigma(t_1)=\sigma(t_2)$ there is a single connected component of closed curves allowing a kink. The bound on absolute curvature is so rigid that a curve with terminals $\sigma(t_1),\sigma(t_2) \in \partial D_1\cap \partial D_2$ being entirely defined in $D_1\cup D_2$ cannot have a point in $int(D_1\cup D_2)\setminus cl(D_1\cap D_2)$.
\end{observation}

%----------------------------------------------------------------------------------------
%	OBS 2
%----------------------------------------------------------------------------------------
\begin{observation} \label{longarc}
Long curves may zigzag in between terminals or spiral around them, see Fig. \ref{spiralex}. Short curves may zigzag in between terminals but clearly never spiral. If a long curve zigzag in between terminals we split it into several adjacent long curves by considering the intersection of the curve and the line joining the terminals. For example, suppose a long curve intersects the line joining the terminals at a point other than the terminals say at $\sigma(t^*)$, then it is splitted into two adjacent long curves $\sigma([t_1,t^*])$ and  $\sigma([t^*,t_2])$, see Fig. \ref{spiralex}. 

In addition, a long curve that spiral around its terminals always admit a point $\sigma(t^*)$ in the line connecting its terminals such that $\sigma([t_1,t_2])\subset\sigma([t_1,t^*])$. 
\end{observation}

%----------------------------------------------------------------------------------------
%	FIGURE 3
%----------------------------------------------------------------------------------------
\begin{figure}[h]
	\centering
	\includegraphics[width=1\textwidth,angle=0]{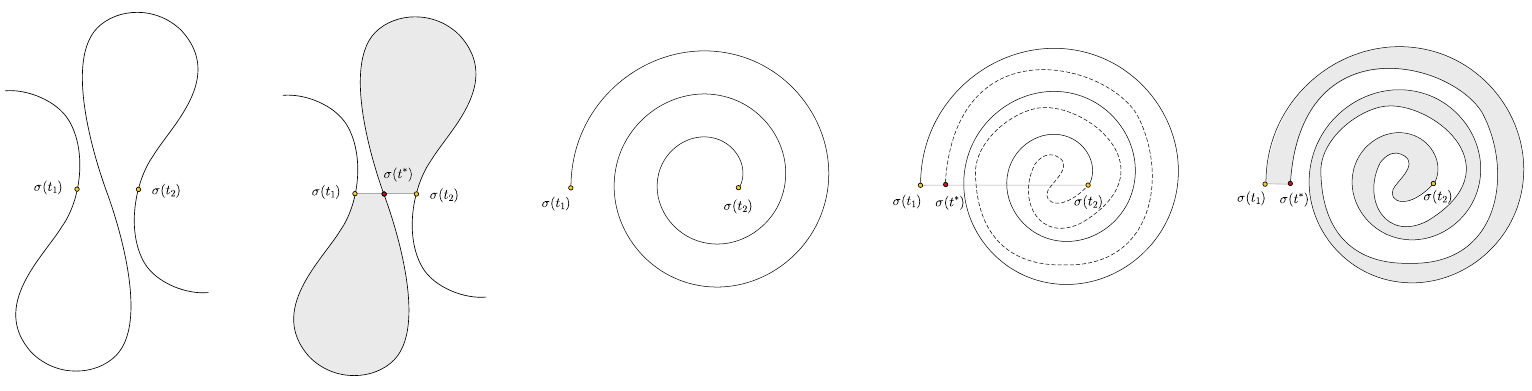}
\caption{A zigzag can be decomposed into adjacent long curves. A long curve that spiral around its terminals can be completed to a longer curve.}
	\label{spiralex}
\end{figure}

%----------------------------------------------------------------------------------------
%	DEFN ARC
%----------------------------------------------------------------------------------------
\begin{definition}\label{dfnarc1} A $\kappa$-constrained curve $\sigma:I\to \mathbb R^2$ is called an arc if the region enclosed by its image and the line connecting its terminal $\sigma(\partial I)$ is homeomorphic to a disk. An arc with terminals removed is called open.
\end{definition}

%NOT SURE ABOUT THIS
%To facilitate exposition we consider any short curve to be a short arc.

\newpage
%----------------------------------------------------------------------------------------
%	EXISTS LONG ARC
%----------------------------------------------------------------------------------------
\begin{theorem} \label{dfnlong} A $\kappa$-constrained loop admits a long arc. 
\end{theorem}

\begin{proof} Every $\kappa$-constrained loop admit a short arc since terminals can be chosen to be arbitrarily close. Consider a short arc $\alpha:[0,s]\to \mathbb R^2$ in an arc length parametrised loop $\sigma:[0,\ell]\to \mathbb R^2$ with $\sigma(0)=\sigma(\ell)$. Suppose that this loop does not admit a long arc. In this case, the terminals of $\alpha$ split the loop into two complementary short arcs $\alpha:[0,s]\to \mathbb R^2$, $\beta:[s,\ell]\to \mathbb R^2$.
 
Since $\alpha$ is short, Theorem 4.21 in \cite{papere} implies that $\alpha([0,s])\subset  D_1\cap D_2$ where $D_1, D_2$ are the radius $r$ disks with $\alpha(0), \alpha(s)\in \partial D_1\cap \partial D_2$, see Fig. \ref{endpartan} left. Since $\beta$ is also short, Theorem 4.21 in \cite{papere} implies that $\beta([s,\ell])\subset  D_1\cap D_2$ with $\beta(s), \beta(\ell)\in \partial D_1\cap \partial D_2$. Let $L_1$ and $L_2$ be the lines tangent to $\partial D_1$ and $\partial D_2$ at $\alpha(s)=\beta(s)$ respectively, see Fig. \ref{endpartan}. These lines divide the plane into four quadrants. Since $\alpha([0,s])\subset  D_1\cap D_2$ the bound on curvature implies that $\alpha'(s)$ can only range in quadrant I. On the other hand, since $\beta([s,\ell])\subset  D_1\cap D_2$, we have that $\beta'(s)$ can only range in quadrant III. We conclude that $\alpha'(s)\neq \beta'(s)$ implying that $\sigma'(s)$ is not of class $C^1$ contradicting the smoothness of $\sigma$. Therefore $\beta$ must have a point in $(D_1\cap D_2)^c$ concluding that is a long curve. We conclude the proof by noting that this long curve can be splitted or completed as in Observation \ref{longarc} to obtain the desired long arc.
\end{proof}

%----------------------------------------------------------------------------------------
%	FIGURE 4
%----------------------------------------------------------------------------------------
\begin{figure}[h]
	\centering
	\includegraphics[width=1\textwidth,angle=0]{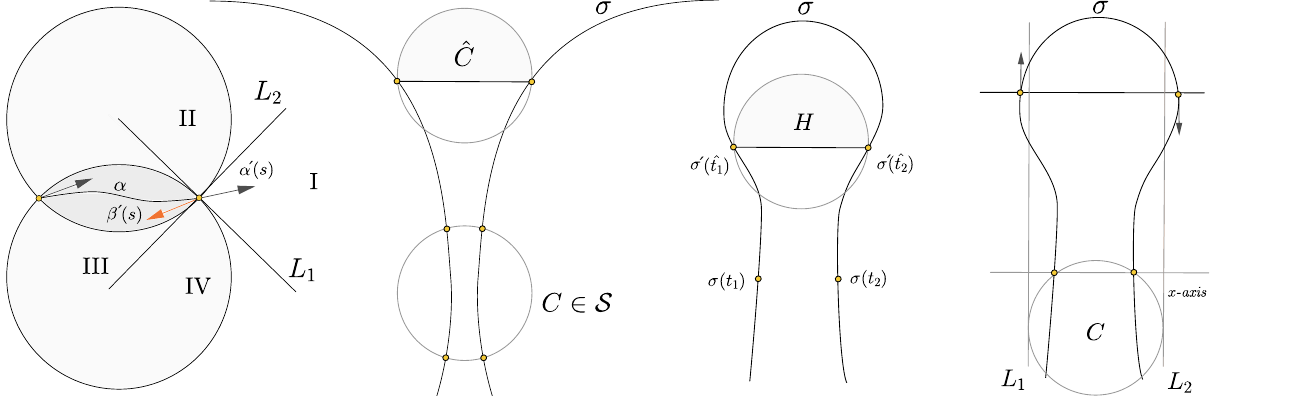}
\caption{From left to right. The notation for Theorem \ref{dfnlong}. A long arc with distinguished end characterised by $\hat C\in \mathcal S$ in Theorem \ref{maxterminal}. An end does not intersect the half disk $H$ supporting its terminals, otherwise it would admit essential terminals. An illustration of a cross section.}
	\label{endpartan}
\end{figure}

%----------------------------------------------------------------------------------------
%	DEFN END
%----------------------------------------------------------------------------------------
\begin{definition}\label{end} If $\sigma:( t_1, t_2)\subset S^1\to \mathbb R^2$ does not admit essential terminals, then $\sigma:[ t_1, t_2]\to \mathbb R^2$ is called an end.  An end is called essential if this is a subset of an arc whose terminals are essential. An end that is not essential is called inessential. 
\end{definition}

%----------------------------------------------------------------------------------------
%	A LONG ARC ADMITS AN END
%----------------------------------------------------------------------------------------
\begin{theorem}\label{maxterminal}  A long arc admits an end. The terminals of an end are distant apart $2r$.
\end{theorem}
\begin{proof}   Let $\sigma: [t_1^*,t_2^*] \rightarrow {\mathbb R}^2$ be a long arc. 

Let $\mathcal S$ be the set of radius $r$ circles $C\subset \mathbb R^2$ such that $\sigma(t_1),\sigma(t_2)\in C$ are the terminals of an arc $\sigma: [t_1,t_2] \rightarrow {\mathbb R}^2$ satisfying:
\begin{itemize}
\item $||\sigma(t_1)-\sigma(t_2)||\leq 2r$;
\item $\sigma([t_1,t_2])\subset \sigma([t_1^*,t_2^*])$ has a point above $C$.
\end{itemize}

It is easy to see that $\mathcal S$ is closed and bounded considered as a subset of the plane consisting of the centres of the circles $C \in \mathcal S$. Due to compactness, there exists a circle $\hat C \in \mathcal S$ such that $\sigma(\hat t_1),\sigma(\hat t_2)\in \hat C$ are the terminals of an arc $\sigma: [\hat t_1,\hat t_2] \rightarrow {\mathbb R}^2$ having the smallest length amongst all the circles in $\mathcal S$, see Fig. \ref{endpartan}. If $\sigma: (\hat t_1,\hat t_2) \rightarrow {\mathbb R}^2$ admits essential terminals, then the minimality of the length of $\sigma: [\hat t_1,\hat t_2] \rightarrow {\mathbb R}^2$ is contradicted, since the existence of essential terminals in $\sigma: (\hat t_1,\hat t_2) \rightarrow {\mathbb R}^2$ would imply the existence of a shorter long arc. We conclude that $\sigma: [t_1,t_2] \rightarrow {\mathbb R}^2$ admits an end.

Suppose that the terminals of an end $\sigma(\hat t_1),\sigma(\hat t_2)\in \hat C$ satisfy $||\sigma(\hat t_1)-\sigma(\hat t_2)||< 2r$. Then a small perturbation of $\hat C$ would reduce the length of the end contradicting the minimality of the end. The terminals of an end are antipodes in $\hat C\in \mathcal S$, concluding the proof.
\end{proof}

%----------------------------------------------------------------------------------------
%	A LOOP ADMITS AN END
%----------------------------------------------------------------------------------------
\begin{corollary}\label{maxterminal2}  A $\kappa$-constrained loop admits an end.
 \end{corollary}
 \begin{proof} By Theorem  \ref{dfnlong} a $\kappa$-constrained loop admits a long arc. The result follows by Proposition \ref{maxterminal}.
 \end{proof}
 
Ends are objects that uniquely characterise families of long arcs. Suppose that $\sigma: [\hat t_1,\hat t_2] \rightarrow {\mathbb R}^2$ is an end. Then this end may also be the end of the family of long arcs $\sigma: [\hat t_1-\epsilon,\hat t_2+\epsilon] \rightarrow {\mathbb R}^2$ for some $\epsilon>0$. In particular, $\sigma([\hat t_1,\hat t_2])\subset \sigma([\hat t_1-\epsilon,\hat t_2+\epsilon])$. 
 
%----------------------------------------------------------------------------------------
%	A LONG ARC HAS LENGTH AT LEAST PI
%----------------------------------------------------------------------------------------
\begin{theorem} \label{dfnlonglength}  A long arc has length at least $\pi r$. 
\end{theorem}

\begin{proof} By Theorem \ref{maxterminal} a long arc admits an end. The result follows after proving that an end has length at least $\pi r$. Consider a long arc whose end $\sigma :[0,\ell] \rightarrow {\mathbb R}^2$ has been reparametrised by its arc length given by $\sigma (t)=(-\rho(t)\cos \theta(t), \rho(t) \sin \theta(t))$ with $\rho,\theta:[0,\ell]\to \mathbb R^2$.

By Theorem \ref{maxterminal} the terminals of an end satisfy $||\sigma(0)-\sigma(\ell)||=2r$. Set a coordinate system in which the $x$-axis passes through $\sigma(0)=(-r,0)$ and $\sigma(\ell)=(r,0)$ with the origin being the midpoint between them. We obtain that $\rho(0)=\rho(\ell)=r$ and $\theta(0)=0$ and $\theta(\ell)=\pi$.
Note that,
$$\int_0^{\ell}\,||\sigma'(t)||\,dt =\int_0^{\ell} \sqrt{{||\rho'(t)||^2+ \rho(t)^2}\,||\theta'(t)||^2}\,\, dt \geq  \int_0^{\ell} { \rho(t)||{\theta}'(t)||}\,\, dt.$$

Let $H$ be the radius $r$ half disk with nonnegative ordinate centred at the origin. Clearly, the end cannot intersect a point of positive ordinate in $\partial H$ unless the end coincides with the semicircle $\partial H$. If so, $\sigma$ would admit essential terminals and not longer be an end, see Fig. \ref{endpartan}. We conclude that $\rho(t)\geq r$ for points in the upper half plane. Since $\theta(\ell)=\pi $ we have that,

$$ \int_0^{\ell}\,||\sigma'(t)||\,dt \geq \int_0^{\ell} { \rho(t)||{\theta}'(t)||}\,\, dt \geq \pi r.$$
\end{proof}

%----------------------------------------------------------------------------------------
%	HALF PESTOV IONIN
%----------------------------------------------------------------------------------------
The next result is an inclusion principle for long $\kappa$-constrained arcs. This can be view as a half of Pestov-Ionin theorem \cite{pestov}.

\begin{theorem}\label{halfdisk} (Half disk). The region enclosed by a long arc and the line segment connecting its terminals encloses a half disk of radius $r$.
\end{theorem} 

\begin{proof}By Theorem \ref{maxterminal} a long arc admits an end. In the proof of Theorem \ref{dfnlonglength} we determined that the end cannot intersect $\partial H$ in a point other than its terminals unless the end coincides with the semicircle $\partial H$. In addition, the end cannot intersect the points in $\partial H$ of zero ordinate other than its terminals, otherwise essential terminals are obtained, contradicting the definition of end. We conclude that $H$ in Theorem  \ref{dfnlonglength} is the desired half disk.  
\end{proof}

%----------------------------------------------------------------------------------------
%	A K-LOOP HAS LENGTH AT LEAST 2PI
%----------------------------------------------------------------------------------------
\begin{corollary} \label{2halfpestov} A $\kappa$-constrained loop has length at least $2\pi r$.
\end{corollary}

\begin{proof} Suppose the loop admits a pair of essential terminals. Then, we apply Theorem \ref{dfnlonglength} to each of the two long arcs with common essential terminals to conclude that the loop has length at least $2\pi r$. 

If the loop admits only inessential terminals, the boundary of a radius $r$ disk $D$ tangent to the loop cannot intersect the loop transversally. Otherwise, a pair of essential terminals is obtained. Consider the centre of $D$ to be the origin. We proceed similarly as in Theorem \ref{dfnlonglength}. Consider the arc length parametrised loop $\sigma :[0,\ell] \rightarrow {\mathbb R}^2$ in polar coordinates $\sigma (t)=(-\rho(t)\cos \theta(t), \rho(t) \sin \theta(t))$ with $\rho,\theta:[0,\ell]\to \mathbb R^2$. Since $D$ is enclosed by the loop we have that $\rho(t)\geq r$. Since $\theta(\ell)=2\pi$ we conclude that 
$$\int_0^{\ell}\,||\sigma'(t)||\,dt \geq  \int_0^{\ell} { \rho(t)||{\theta}'(t)||}\,\, dt\geq 2\pi r.$$
\end{proof}

%----------------------------------------------------------------------------------------
%	A K-LOOP HAS FINITELY MANY ESS ENDS
%----------------------------------------------------------------------------------------
\begin{theorem} \label{ends} A $\kappa$-constrained loop has finitely many essential ends and infinitely many inessential ends.
\end{theorem}

\begin{proof}  A loop admitting no essential terminals has an empty set of essential ends being this finite of cardinality zero.  

By Theorem \ref{dfnlong} a loop admits a long arc. If the terminals of the long arc are essential, then two long arcs share these terminals. By applying Theorem \ref{maxterminal} to both arcs, we establish the existence of distinct ends. By Theorem \ref{dfnlonglength} a long arc has length at least $\pi r$. Therefore, there must be a finite number of essential ends. Otherwise, the rectifiability of the loop would be contradicted. 

If the long arc obtained by Theorem \ref{dfnlong} is inessential its terminals are also the terminals of a short arc. In addition, every pair of inessential terminals have unique inessential end. Since loops are parametrised in a continuum, there are infinitely many choices of inessential terminals, concluding the proof.
\end{proof}

%----------------------------------------------------------------------------------------
%	DEFN PARTAN
%----------------------------------------------------------------------------------------
\begin{definition}\label{defnpartan}
An arc $\sigma :I \rightarrow {\mathbb R}^2$ has {parallel tangents} if there exist $t_1,t_2 \in I$, with $t_1<t_2$, such that $\sigma'(t_1)$ and $\sigma'(t_2)$ are parallel and pointing in opposite directions. 
\end{definition}

%----------------------------------------------------------------------------------------
%	DEFN CROS SECTION
%----------------------------------------------------------------------------------------
\begin{definition}\label{defncross} Let $L_1$ and $L_2$ be the lines $x=-r$ and $x=r$ respectively.  A line joining two points in $\sigma: [t_1,t_2] \rightarrow {\mathbb R}^2$ distant apart at least $2r$ one to the left of $L_1$ and the other to the right of $L_2$ is called a cross section, see Fig. \ref{endpartan} right. 
\end{definition}

Next we prove the existence of a cross section and parallel tangents as in Corollary 3.4 in \cite{papere}, but here using the existence of ends.

%----------------------------------------------------------------------------------------
%	EXISTS CROSSECT, PARTAN, LENGTH >2R
%----------------------------------------------------------------------------------------
\begin{corollary} \label{crossect} A long arc has a cross section, parallel tangents and length of at least $2r$.
\end{corollary}

\begin{proof} By Theorem \ref{maxterminal} a long arc admits and end $\sigma:[\hat t_1,\hat t_2] \to \mathbb R^2$. Since $||\sigma(\hat t_1)-\sigma(\hat t_2)||=2r$ the long arc admits a cross section. Therefore, the end has length at least $2r$ and so the long arc containing it. 

According the proof of Theorem \ref{maxterminal} the terminals of an end are antipodes in $\hat C\in \mathcal S$ and the end is of minimal length amongst all the choices for terminals in $C\in \mathcal S$. Let $L$ be the line joining the terminals and $L_1,L_2$ be the perpendicular lines to $L$ at each terminal. The maximality of $\hat C$ implies that locally the trajectory of the end after the first terminal must be to the left of $L_1$ (or coincides with $L_1$) and the trajectory of the end before the second terminal must be to the right of $L_2$ (or coincides with $L_2$). Otherwise, we find essential terminals, leading to a contradiction. 

If the trajectories coincide with the parallels, then there exist parallel tangents. If none of the trajectories coincides with the parallels, by continuity, the end intersects $L_1$ and $L_2$ at points other than the terminals, see Fig. \ref{endpartan} right. By the mean value theorem, there exist $t_1, t_2\in [\hat t_1,\hat t_2]$ such that $\sigma'(t_1)$ and $\sigma'(t_2)$ are parallel pointing in opposite directions. The case when only one trajectory coincides with a parallel runs similarly.
\end{proof}

%----------------------------------------------------------------------------------------
%	OBS INEQUALITIES
%----------------------------------------------------------------------------------------
\begin{observation}\label{ccomp}Let $\sigma:S^1\to \mathbb R^2$ be a $\kappa$-constrained loop. In Corollary \ref{2halfpestov} we proved that length$(\sigma) \geq 2\pi r$. In Theorem \ref{dfnlong} we proved that $\sigma$ admits a long arc.  In Theorem \ref{maxterminal} we proved that a long arc admits an end. Ends by definition are proper subsets of long arcs. In Theorem \ref{dfnlonglength} we proved that $\mbox{length(long arc)}\geq \pi r$. From Corollary \ref{crossect} we have that a long arc admits a cross section with $\mbox{length(cross section)}\geq 2r$. We derive the following:
$$\mbox{length($\sigma$})>\mbox{length(long arc)}>\mbox{length(end)}>\mbox{length(cross section)}.$$
\end{observation}

The next result provides a characterisation of convexity through curvature. A regular closed plane curve is convex if and only if it is the boundary of a convex set $\mathcal K\subset \mathbb R^2$ \cite{docarmo}. Next we assert on the convexity of  $\partial \mathcal K$ to conclude on the convexity of $\mathcal K$. 

%----------------------------------------------------------------------------------------
%	FIGURE 5
%----------------------------------------------------------------------------------------
\begin{figure}[h]
	\centering
	\includegraphics[width=.9\textwidth,angle=0]{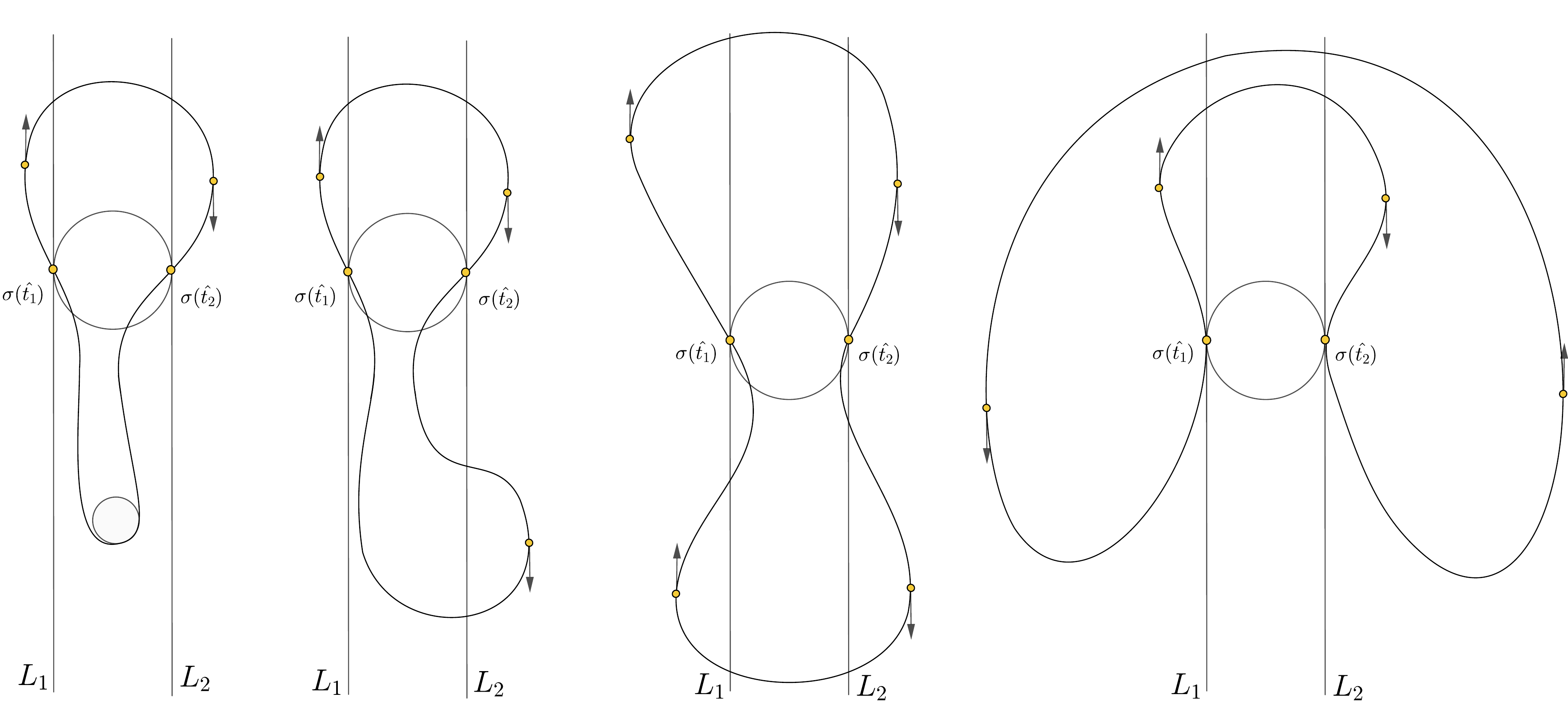}
\caption{ }
	\label{essnotconv}
\end{figure}

%----------------------------------------------------------------------------------------
%	THM ESS NOT CONVEX - CONVEX INESS
%----------------------------------------------------------------------------------------
\begin{theorem}\label{nonconvex} If $\mathcal K\subset \mathbb R^2$ is a domain with boundary $\partial \mathcal K=\sigma$ a $\kappa$-constrained loop admitting a pair of essential terminals. Then, $\mathcal K$ is not convex. If $\mathcal K\subset \mathbb R^2$ is a convex domain with boundary $\partial \mathcal K=\sigma$ a $\kappa$-constrained loop. Then, $\partial \mathcal K=\sigma$ admits only inessential terminals.
\end{theorem} 
\begin{proof}  
 
 Since a regular simple closed convex plane curve has always two and only two parallel tangents to any given direction \cite{emech} we contradict convexity by searching for a third one. We borrow the notation from Corollary \ref{crossect}. 

Consider a long arc with essential terminals. By Corollary \ref{crossect} this long arc admits two parallel tangents located somewhere at an end whose existence is guaranteed by Theorem \ref{maxterminal}. Next, we analyse the trajectory for the loop before and after the terminals of the end. Let $\sigma(\hat t_1-\epsilon)$ and $\sigma(\hat t_2+\epsilon)$ be points sufficiently close to the terminals of the end, $\epsilon>0$, see Fig. \ref{essnotconv}. We have some cases:

\begin{enumerate}
\item  $\sigma(\hat t_1-\epsilon), \sigma(\hat t_2+\epsilon)$ lie in between $L_1$ and $L_2$. 
\item  $\sigma(\hat t_1-\epsilon)$ lies to the right of $L_1$ and $ \sigma(\hat t_2+\epsilon)$ lies to the right of $L_2$. 
\item  $\sigma(\hat t_1-\epsilon)$ lies to the left of $L_1$ and $\sigma(\hat t_2+\epsilon)$ lies to the left of $L_2$.
\item   $\sigma(\hat t_1-\epsilon)$ lies to the left of $L_1$ and $ \sigma(\hat t_2+\epsilon)$ lies to the right of $L_2$.
%\item  $\sigma(\hat t_1-\epsilon)$ lies in $L_1$ or $ \sigma(\hat t_2+\epsilon)$ lies in $L_2$.
\end{enumerate}

Proof of (1). Note these are the terminals of two long arcs from one of which we already obtained a pair of parallel tangents, see Fig. \ref{essnotconv} left. By Lemma \ref{r1} the second long arc cannot be entirely defined in the open band in between $L_1$ and $L_2$. Therefore, the loop must abandon the band and reenter it. Suppose the loop leaves and reenters the band intersecting $L_2$ twice, see second illustration in Fig. \ref{essnotconv}. By the mean value theorem, there exists another pair of parallel tangents implying that the loop has three parallels tangents, concluding that it cannot be the boundary of a convex set. The case where the loop leaves and reenters the band intersecting $L_1$ twice is analogous. If the loop intersects $L_1$ and $L_2$ we obtain four parallel tangents, so it cannot be the boundary of a convex set, see third illustration in Fig.\ref{essnotconv}.

The cases (2)-(3) run similarly. The case (4) is illustrated in Fig.\ref{essnotconv} right an also runs similarly. In general, any possible way of connecting the terminals $\sigma(\hat t_1-\epsilon), \sigma(\hat t_2+\epsilon)$ would lead to at least a third parallel tangent, concluding the proof.
\end{proof}

%----------------------------------------------------------------------------------------
%	INTERNAL-EXTERNAL  REGIONS
%----------------------------------------------------------------------------------------
\begin{definition}
Consider the domain delimited by a long arc and the line segment connecting its terminals. If this is included in the bounded plane region separated by a $\kappa$-constrained loop we call it internal. If this is included in the unbounded region separated by a $\kappa$-constrained loop we call it external. Certain domains are neither internal nor external.
\end{definition}

The next result gives necessary and sufficient conditions for the existence of rolling disks for plane not necessarily convex domains whose boundary are $\kappa$-constrained curves.

%----------------------------------------------------------------------------------------
%	OBS VALID INTERSECTION ROLLING
%----------------------------------------------------------------------------------------
\begin{observation}\label{obsintersect} If $\mathcal K$ is internal rolling, then $\partial D(t)\cap \partial \mathcal K$ with $D(t)\subset cl(\mathcal K)$ can only be a singleton, a pair of antipodal points in $D(t)$, or an arc in $D(t)$, see Definition \ref{dfnrolling}.
\end{observation}
%Internal rolling may not be convex domains.

%THE BEAUTY OF THIS THEM IS THAT ROLLINGNESS IS DETECTED ONLY BY LOCAL DATA
%----------------------------------------------------------------------------------------
%	MAIN THEOREM - IFF
%----------------------------------------------------------------------------------------
\begin{theorem} \label{main} Let $\mathcal K\subset \mathbb R^2$ be a domain with boundary $\partial \mathcal K=\sigma$ a $\kappa$-constrained loop. Then,
\begin{enumerate}
\item $ \mathcal K$ is internal rolling if and only if $\partial \mathcal K$ does not admit a pair of essential terminals for complementary internal long arcs.
\item $\mathcal K$ is external rolling if and only if $\partial \mathcal K$ does not admit a pair of essential terminals for complementary external long arcs one bounded and the other unbounded.
\item $\mathcal K$ is rolling if and only if $ \mathcal K$ admits only inessential terminals. Here $\mathcal K$ is not necessarily convex.
\end{enumerate}
\end{theorem} 

\begin{proof} (1) Suppose $\mathcal K$ is internal rolling. Them, for each $\sigma(t)\in \partial \mathcal K$  there exists a radius $r$ disk $D(t)\subset cl(\mathcal K)$. Suppose $\partial \mathcal K$ admits a pair of essential terminals for complementary internal long arcs. Since the distance between essential terminals is less than $2r$ we have that no disk tangent to an essential terminal can be included in $cl(\mathcal K)$ leading to a contradiction. 

Conversely, suppose $\partial \mathcal K$ does not admit a pair of essential terminals for complementary internal long arcs. If $ \mathcal K$ is not internal rolling, then there exists a radius $r$ disk $D(t)$ tangent to $\partial \mathcal K$ at $\sigma(t)$ such that $\partial D(t)\cap \partial \mathcal K$ have an element different from the ones in Observation \ref{obsintersect}. Let $\sigma(t^*)$ be such an element. Since $D(t)$ is tangent at $\sigma(t)$ Lemma \ref{r1} implies that the arc $\sigma:[t,t^*]\to \mathbb R^2$ as a point above $\partial D(t)$ implying this arc is long, concluding that $\sigma(t),\sigma(t^*)$ are essential terminals, leading to a contradiction.

(2) runs similarly as (1) after replacing internal by external. The existence of an unbounded long arc comes from the non compactness of $\mathbb R^2$. In Fig. \ref{exroll} left we show a non external rolling domain admitting a pair of essential terminals, in this case, the line joining the terminals decompose the plane into one external region homeomorphic to a disk and one unbounded region. Note that by adding a point at the infinite to $\mathbb R^2$ the unbounded long arc becomes bounded. 

(3) Since $\mathcal K$ is rolling if and only if it is internal and external rolling the proof follows by combining (1) and (2). Note that sets admitting only inessential terminals are not necessarily convex.
\end{proof}

We present an updated version for the Blaschke rolling disk theorem \cite{blaschke} for $C^1$ and piecewise $C^2$ simple closed curves satisfying a bound on absolute curvature.

\begin{theorem} Suppose that a convex domain $\mathcal K\subset \mathbb R^2$ has boundary $\partial \mathcal K=\sigma$ a $\kappa$-constrained loop. Then $\mathcal K$ is rolling.
\end{theorem}
\begin{proof} By Theorem \ref{nonconvex} a convex domain with boundary a $\kappa$-constrained loop  admits only inessential terminals. The result follows by (3) in Theorem \ref{main}.
\end{proof}

%----------------------------------------------------------------------------------------
%	MAXIMAL ROLLING DISK DECOMPOSITION
%----------------------------------------------------------------------------------------
\section{Rolling disk decomposition} \label{algorithm}

In Theorem \ref{main} we established that a domain $\mathcal K\subset \mathbb R^2$ is rolling if and only if $\partial \mathcal K$ does not admit essential terminals. Next, we propose a method for obtaining a finite collection of maximal (internal or external) rolling regions for a domain that is not necessarily internal or external rolling.

%----------------------------------------------------------------------------------------
%	BOUNDED UNBOUNDED REGIONS
%----------------------------------------------------------------------------------------
%\begin{definition} \label{cdpjordan} The bounded and unbounded domains enclosed a $\kappa$-constrained loop are denoted by $\Lambda$ and $\Lambda^*$ respectively. 
%\end{definition}

%----------------------------------------------------------------------------------------
%	 DECOMPOSITION THEOREM
%----------------------------------------------------------------------------------------
\begin{theorem}\label{dcp} Suppose a domain $\mathcal K\subset \mathbb R^2$ has boundary $\partial \mathcal K=\sigma$ a $\kappa$-constrained loop. Then there exists a decomposition of $\mathbb R^2 \setminus \partial \mathcal K$ into a finite number of maximal (internal or external) rolling regions and certain excluded regions.
\end{theorem} 

\begin{proof}  Suppose $\partial \mathcal K=\sigma$ is a loop with only inessential terminals. Then, Theorem \ref{main} implies that $\mathcal K$ is rolling. The decomposition of $\mathbb R^2 \setminus \partial \mathcal K$ consists exactly on the two maximal rolling domains (with respect set inclusion) separated by the loop.

Suppose $\partial \mathcal K=\sigma$ is a loop admitting a pair of essential terminals. 

\begin{enumerate}

\item {\bf Rolling algorithm}: We proceed to construct a domain such that for each element on its boundary there exist disk of radius $r={1}/{\kappa}$ tangent to the boundary included in the domain.

By the Pestov-Ionin theorem \cite{pestov} the bounded region delimited by the loop encloses a radius $r$ disk $D_0$ that can be chosen to be tangent to $\partial \mathcal K$. With Definition \ref{dfnrolling} in mind, proceed to roll clockwise $D_0$ until it first intersects the loop at a point other than the tangent point at $\partial \mathcal K$. If the first intersecting point is one of the types described in Observation \ref{obsintersect}, we continue the rolling process until we obtain a pair of essential terminals tangent to the boundary of a rolling disk, see Fig. \ref{rollalg}. Let $D_1$ be the rolling disk containing such essential terminals on its boundary and let $\lambda_1$ be the shortest arc in $\partial D_1$ connecting these terminals.  Replace one of the two long arcs whose essential terminals are the ones detected by the arc $D_1$. The long arc replaced is the one having a point above $\lambda_1$, see Lemma \ref{r1}. 

We continue this rolling and replacing process until cyclically we return to the starting replacement $\lambda_1=\lambda_{k+1}$. Since each of the detected essential terminals characterise a unique long arc and that loops are rectifiable, the replacing process must be finite. Let $\mathcal R_1$ be the domain enclosed by the cyclic concatenation of the portions of $\partial \mathcal K$ rolled by $\{D_{j}\}_{j=1}^k$ together with the replacing arcs $\{\lambda_{j}\}_{j=1}^k$ see Fig. \ref{rollalg} right. Clearly, $\mathcal R_1$ is a maximal domain with respect set inclusion. Also, the piecewise smooth curve traced by the center of the rolling disks uniquely characterises $\mathcal R_1$.

%Note the curve traced by the centres of the rolling disks is piecewise smooth.
%----------------------------------------------------------------------------------------
%	FIGURE 6
%----------------------------------------------------------------------------------------
\begin{figure}[h]
	\centering
	\includegraphics[width=1\textwidth,angle=0]{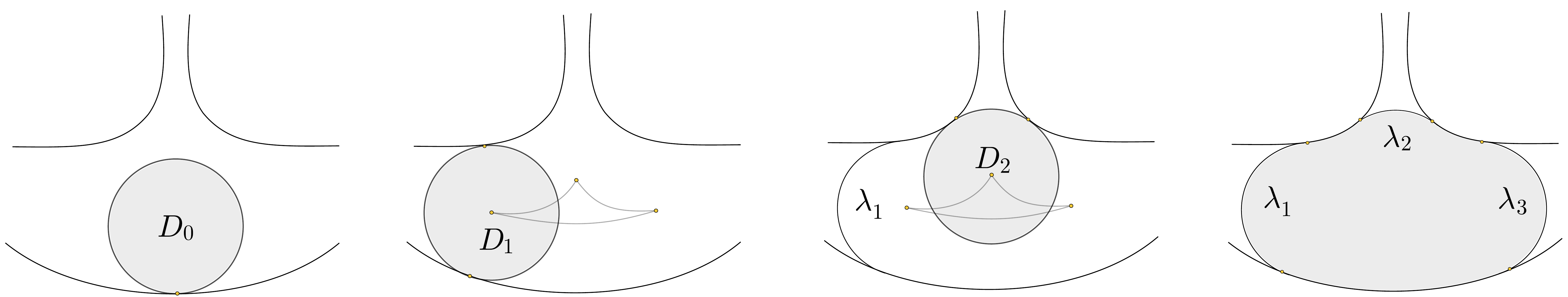}
\caption{An example of the rolling algorithm for $k=3$. The grey region depicts $\mathcal R_1$. }
	\label{rollalg}
\end{figure}

\item  {\bf Necks}: In the previous step we exchanged $k$ long arcs by $k$ (replacing) arcs. Now, we concentrate on each of the replaced long arcs. Consider the rolling disk $D_j$ containing the replacing arc $\lambda_j$ whose terminals $\sigma(t_{1j}),\sigma(t_{2j})$ are also the (essential) terminals of the replaced long arc, see Fig. \ref{rollalg1} left. 

%lying in the boundary of a rolling disk containing the terminals of the respective long arc.
%having a point above the line connecting the terminals.

By continuously sweeping  along the long arc $\sigma: [t_{1j},t_{2j}]\to \mathbb R^2$ with a family of circles $C\in \mathcal S$ (similarly as in Theorem \ref{maxterminal}) we determine the existence of an end, see Corollary \ref{maxterminal2}. Let $D_j^0$ be a radius $r$ disk tangent to one of the terminals of the end, see Fig. \ref{rollalg1}. The case where $D_j^0$ intersects $\partial \mathcal K=\sigma$ transversally is treated in (3). Roll $D_j^0$ towards the line connecting the terminals of the end to obtain a pair of essential terminals. Let $D_j^1$ be the radius $r$ disk containing on its boundary such terminals denoted by $\sigma(t_{3j}),\sigma(t_{4j})$.

%----------------------------------------------------------------------------------------
%	FIGURE 7
%----------------------------------------------------------------------------------------
\begin{figure}[h]
	\centering
	\includegraphics[width=1\textwidth,angle=0]{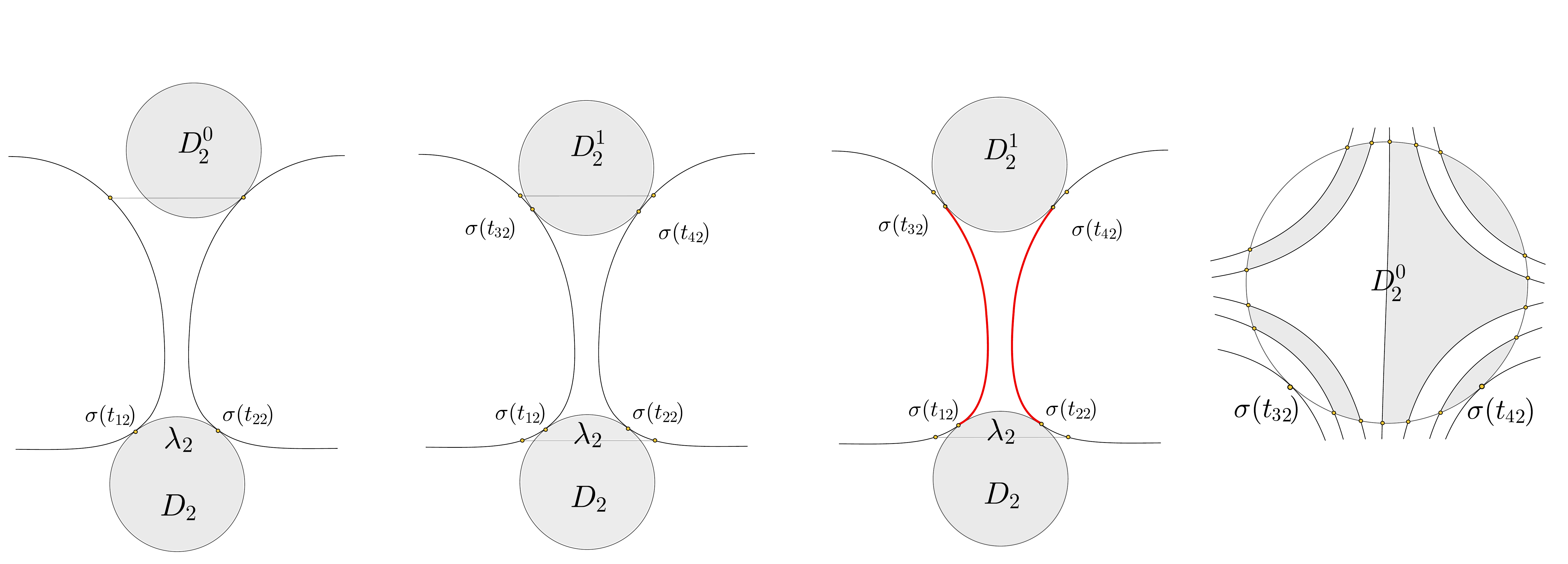}
\caption{An example of the process described in (2) for $j=2$ according to Fig. \ref{rollalg}. The coloured  `parallel arcs' depict a neck. The loop intersects transversally $\partial D_2^0$ at several points. }
	\label{rollalg1}
\end{figure}

Due to the continuous circle sweeping each of the arcs $\sigma: [t_{1j},t_{3j}]\to \mathbb R^2$ and $\sigma: [t_{4j},t_{2j}]\to \mathbb R^2$ may be short or simple but never long. Let a neck be conformed by the two `parallel arcs' in a loop obtained as above. In this case, we obtained the neck $\sigma([t_{1j},t_{3j}]) \cup \sigma([t_{4j},t_{2j}])$, see third illustration in Fig. \ref{rollalg1}. Note that a neck is a set of `parallel essential terminals'. The regions enclosed by a neck and their respective replacement arcs are clearly (excluded) not rolling domain.

%Note a neck is defined in between ends of complementary long arcs.  
%

\item {\bf Disk intersections}: If $D_j^0$ intersects transversally the loop, then $D_j^0$ encloses a portion on a (excluded) not rolling domain. In this case, we consider pairs of elements in $\partial D_j^0 \cap \partial \mathcal K$ and apply to them the turning condition in Definition \ref{dfnarc} to determine arc type, see Fig. \ref{rollalg1} right. Select the essential terminals in $\partial D_j^0 \cap \partial \mathcal K$ and proceed to apply the procedure in (2) to determine their associated necks. Note some excluded regions may have as part of its boundary a replacement arc.

If $D_j^1$ only intersects the loop according to Observation \ref {obsintersect} and at the detected essential terminals. Then, we apply the rolling algorithm in (1) to obtain another rolling domain.

By recursively applying this process we obtain a decomposition of $\mathbb R^2 \setminus \partial \mathcal K$ into a finite number of excluded regions and a finite number of rolling regions. It is easy to see that this decomposition is unique and maximal with respect set inclusion, concluding the proof.
\end{enumerate}
\end{proof}

%

%----------------------------------------------------------------------------------------
%	REFERENCES
%----------------------------------------------------------------------------------------

 \end{document}